\documentclass[a4paper]{article}

\usepackage{RR}
\RRNo{6481}
\usepackage[french]{babel}

\usepackage{epsfig}
\usepackage{amssymb}
\usepackage{epic}
\usepackage{eepic}
\usepackage{graphics}
\usepackage{hyperref}

\usepackage{mathrsfs}
\usepackage[latin1]{inputenc}
\usepackage{fancyhdr}
\usepackage{textcomp}
\usepackage{enumitem}
\usepackage{lineno}

\usepackage{latexsym} % Sï¿½bolos
\usepackage{amsmath,amssymb}
\usepackage{graphicx}

\usepackage{wasysym}

\newcommand{\PartIntSup}[1]{\left\lceil #1\right\rceil}

\usepackage{pb-diagram}

\def\lgem{\hbox{l\kern-.08em\raise .7ex\hbox{.}\kern-.08eml}}
\def\Lgem{\hbox{L\kern-.08em\raise .7ex\hbox{.}\kern-.08emL}}

  \renewenvironment{thebibliography}[1]{%
    \begin{oldthebibliography}{#1}%
      \setlength{\parskip}{0ex}%
      \setlength{\itemsep}{2.0ex}%
  }%
  {%
    \end{oldthebibliography}%
  }

\newtheorem{claim}{Claim}
\newtheorem{definition}{Definition}[section]
\newtheorem{theorem} {Theorem}[section]
\newtheorem{lemma} {Lemma}[section]

\newtheorem{proposition}{Proposition}[section]
\newtheorem{conjecture}{Conjecture}[section]
\newtheorem{corollary}{Corollary}[section]

\newenvironment
{proof}{\begin{trivlist} \item[] {\em \textbf{Proof}: }}{\hfill
$\Box$\\
                       \end{trivlist}}

\def\ben{\begin{enumerate}%
\itemsep 1pt plus 1pt minus 1pt}
\def\een{\end{enumerate}}
\def\bit{\begin{itemize}%
\itemsep 1pt plus 1pt minus 1pt}
\def\eit{\end{itemize}}

\RRtitle{Groupage de Trafic Dans les Anneaux Unidirectionnels WDM avec
Graphe de Requ\^etes de Degr\'e Born\'e \thanks{This work has been
partially supported by European project IST FET AEOLUS, PACA region of
France, Ministerio de Educaci\'on y Ciencia of Spain,  European Regional
Development Fund under project TEC2005-03575, Catalan Research Council
under project 2005SGR00256, and COST action 293 GRAAL, and has been done
in the context of the  {\sc crc Corso} with France Telecom.}}

\RRetitle{Traffic Grooming in Unidirectional WDM Rings with Bounded Degree
Request Graph \thanks{This work has been partially supported by European
project IST FET AEOLUS, PACA region of France, Ministerio de Educaci\'on y
Ciencia of Spain,  European Regional Development Fund under project
TEC2005-03575, Catalan Research Council under project 2005SGR00256, and
COST action 293 GRAAL, and has been done in the context of the  {\sc crc
Corso} with France Telecom.}}

\RRauthor{Xavier Muñoz \thanks{Graph Theory and Combinatorics group, MA4,
UPC, Barcelona, Spain} \and Ignasi Sau \thanks{Mascotte Project - I3S
(CNRS/UNSA) and INRIA - Sophia-Antipolis, France} \thanks{Graph Theory and
Combinatorics group, MA4, UPC, Barcelona, Spain}}

\RRdate{February 2008}

\RRabstract{Traffic grooming is a major issue in optical networks. It
refers to grouping low rate signals into higher speed streams, in order to
reduce the equipment cost. In SONET WDM networks, this cost is mostly
given by the number of electronic terminations, namely ADMs. We consider
the case when the topology is a unidirectional ring. In graph-theoretical
terms, the traffic grooming problem in this case consists in partitioning
the edges of a request graph into subgraphs with a maximum number of
edges, while minimizing the total number of vertices of the decomposition.

We consider the case when the request graph has bounded maximum degree
$\Delta$, and our aim is to design a network being able to support any
request graph satisfying the degree constraints. The existing theoretical
models in the literature are much more rigid, and do not allow such
adaptability. We formalize the problem, and solve the cases $\Delta=2$
(for all values of $C$) and $\Delta = 3$ (except the case $C=4$). We also
provide lower and upper bounds for the general case.}

\RRresume{Le groupage de trafic est un des probl\`emes les plus importants
dans les r\'eseaux optiques. Il constiste \`a grouper des signaux de bas
d\'ebit dans des flux de plus grande capacit\'e, avec l'objectif de
r\'eduire le co\^ut du r\'eseau. Dans les r\'eseaux SONET WDM, ce co\^ut
est donn\'e principalement par le nombre d'ADMs. Nous consid\'erons
l'anneau unidirectionnel comme graphe physique. En termes de th\'eorie des
graphes, le groupage de trafic revient \`a trouver une partition des
ar\^etes du graphe de requ\^etes en sous-graphes avec un nombre maximum
d'ar\^etes, en minimisant le nombre total de sommets dans la partition.

Nous consid\'erons un graphe de requ\^etes de degr\'e maximum $\Delta$, et
le but est de concevoir un r\'eseau qui soit capable de satisfaire
\emph{tous} les graphes de requ\^etes tels que chaque sommet peut
\'etablire au plus $\Delta$ communications. Un mod\`ele permettant cette
flexibilit\'e n'existait pas dans la litt\'erature. Nous formalisons le
probl\`eme et trouvons la solution exacte pour $\Delta =2$ et $\Delta =3$
(sauf le cas $C=4$). Nous donnons des bornes inf\'erieures et
sup\'erieures pour le cas g\'eneral. }

 \RRkeyword{Optical networks, SONET over WDM, traffic grooming,
ADM, graph decomposition, cubic graph, bridgeless graph.}

\RRmotcle{R\'eseaux optiques, SONET, groupage de trafic, ADM,
d\'ecomposition des graphes, graphe cubique, graphes sans bridges.}

\RRprojet{MASCOTTE} \RRtheme{\THCom} \URSophia

\begin{document}
\makeRR

\section{Introduction}
%%%%%%%%%%%

%%%%%%%%%%%

Traffic grooming is the generic term for packing low rate signals into
higher speed streams (see the
surveys~\cite{BeCo06,DuRo02b,MoLi01,Som01,ZhMu03}). By using traffic
grooming, it is possible to bypass the electronics at the nodes which are not
sources or destinations of traffic, and therefore reducing the cost of the
network. Typically, in a WDM (Wavelength Division Multiplexing) network,
instead of having one SONET Add Drop Multiplexer (ADM) on every wavelength
at every node, it may be possible to have ADMs only for the wavelengths
used at that node (the other wavelengths being optically routed without
electronic switching).

The so called traffic grooming problem consists in minimizing the total
number of ADMs to be used, in order to reduce the overall cost of the
network. The problem is easily seen to be \textsc{NP}-complete for an arbitrary
set of requests. See \cite{logg, FMMSZ06_1, APS07b} for hardness and approximation
results of traffic grooming in rings, trees and star networks.

Here we consider unidirectional SONET/WDM ring networks. In that case the
routing is unique and we have to assign to each request between two nodes
a wavelength and some bandwidth on this wavelength. If the traffic is
uniform and if a given wavelength can carry at most $C$ requests, we can
assign to each request at most $\frac{1}{C}$ of the bandwidth.  $C$ is
known as the \emph{grooming ratio} or \emph{grooming factor}. Furthermore
if the traffic requirement is symmetric, it can be easily shown (by
exchanging wavelengths) that there always exists an optimal solution in
which the same wavelength is given to a pair of symmetric requests. Then
without loss of generality we will assign to each pair of symmetric
requests, called a \emph{circle}, the same wavelength. Then each circle
uses $\frac{1}{C}$ of the bandwidth in the whole ring. If the two
end-nodes are $i$ and $j$, we need one ADM at node $i$ and one at node
$j$. The main point is that if two requests have a common end-node, they
can share an ADM if they are assigned the same wavelength.

The traffic grooming problem for a unidirectional SONET ring with $n$ nodes
and a grooming ratio $C$ has been modeled as a graph
partition problem in both \cite{BeCo03} and \cite{GHLO03} when the request
graph is given by a symmetric graph $R$. To a wavelength $\lambda$
is associated a subgraph $B_\lambda \subset R$ in which each edge corresponds to a
pair of symmetric requests (that is, a circle) and each node to an ADM.
The grooming constraint, i.e. the fact that a wavelength can carry at most
$C$ requests, corresponds to the fact that the number of edges
$|E(B_\lambda)|$ of each subgraph $B_\lambda$ is at most $C$. The cost
corresponds to the total number of vertices used in the subgraphs, and the
 objective is therefore to minimize this number.\\

 This problem has been
 well studied  when the network is a unidirectional ring~\cite{BeCo06,BCM03,ChMo00,DuRo02,DuRo02b,GHLO03,
 GLS98,GRS00,MoLi01,WCLF00,WCVM01}.
%\cite{BeCe03,BCC+05,BCLY04,BeCo03,BeCo06,BCM03,ChMo00,DuRo02,DuRo02b,GHLO03,
 %GLS98,GRS00,Hu02,Hu02b,MoLi01,WCLF00,WCVM01,YuFu02,ZhQi00}.
With the all-to-all set of requests, optimal constructions for a given
grooming ratio $C$ were obtained using tools of graph and design theory,
in particular for grooming ratio $C=3,4,5,6$ and $C\geq
N(N-1)/6$~\cite{BeCo06}.

Most of the research efforts in this grooming problem have been devoted to find the minimum
number of ADMs required either for a  given traffic pattern or set of connection requests (typically
uniform all-to-all communication pattern), or either for a general traffic pattern. However in most cases
the traffic pattern has been considered as an input for the problem for placing the ADMs.
In this paper we consider the traffic grooming problem from a different point of view:
Assuming a given network topology it would be desirable to place the minimum number of
ADMs as possible at each node in such a way that they could be configured to handle different traffic patterns
or graphs of requests. One cannot expect to change the equipment of the network each time the traffic requirements change.

Without any restriction in the graph of requests, the number of required ADMs is given by the worst case,
i.e. when the Graph of Requests is the complete graph.  However, in many cases some restrictions on
the graph of requests might be assumed. From a practical point of view, it is interesting to design a network being able
to support any request graph with maximum degree not exceeding a given
constant. This situation is usual in real optical networks, since due
to technology constraints the number of allowed communications for each
node is usually bounded. This flexibility can also be thought from another
point of view: if we have a limited number of available ADMs to place at
the nodes of the network, then it is interesting to know which is the
maximum degree of a request graph that our network is able to support,
depending on the grooming factor. Equivalently, given a maximum degree and
a number of available ADMs, it is useful to know which values of the
grooming factor the network will support.

%FALTA:

%- FLEXIBILITAT TAMBE EN QUANT AL NUMERO DE REQUESTS.

%- PODEM DIR QUIN ES EL GRAU MAXIM Q POT SUPORTAR UNA XARXA EN FUNCIO DEL
%NUMERO D'ADMs I DEL GROOMING FACTOR.\\

The aim of this article is to provide a theoretical framework to design
such networks with dynamically changing traffic. We study the case when
the physical network is given by an unidirectional ring, which is a widely
used topology (for instance, SONET rings). In \cite{Berry} the authors
consider this problem from a more practical point of view: they call
$t$\emph{-allowable} a traffic matrix where the number of circuits
terminated at each node is at most $t$, and the objective is also to
minimize the number of electronic terminations. They give lower bounds on
the number of ADMs and provide some heuristics.

In addition, we also suppose that each pair of communicating nodes
establishes a two-way communication. That is, each pair $(i,j)$ of
communicating nodes in the ring represents two requests: from $i$ to $j$,
and from $j$ to $i$. Thus, such a pair uses all the edges of the ring,
therefore inducing one unity of load. Hence, we can use the notation
introduced in \cite{BCM03} and consider each request as an edge, and then
again the grooming constraint, i.e. the fact that a wavelength can carry
at most $C$ requests, corresponds to the fact that the number of edges
$|E(B_\lambda)|$ of each subgraph $B_\lambda$ is at most $C$. The cost
corresponds to the total number of vertices used in the subgraphs.

Namely, we consider the problem of placing the minimum number of ADMs in
the nodes of the ring in such a way that the network could support
\emph{any} request graph with maximum degree bounded by a constant
$\Delta$. Note that using this approach, as far as the degree of each node
does not exceed $\Delta$, the network can support a wide range of traffic
demands without reconfiguring the electronics placed at the nodes. The
problem can be
formally stated as follows:\\

\noindent \textsc{Traffic Grooming in Unidirectional Rings with
Bounded-Degree Request Graph}
\begin{itemize}
\item[\textbf{\textsc{Input:}}] Three integers $n$, $C$, and $\Delta$.
\vspace{0.1cm}
%\item[\textbf{\textsc{Output:}}] A partition of $E(K_n)$
\item[\textbf{\textsc{Output:}}] An assignment of $A(v)$ ADMs to each node $v \in V(C_n)$, in such a way that \emph{for any request graph} $R$ with maximum degree at most $\Delta$,
it exists a partition of $E(R)$ into
  subgraphs $B_\lambda$, $1\leq \lambda \leq \Lambda$, such that:

  \begin{itemize}
  \item[$(i)$]  $|E(B_\lambda)|\leq C$ for all $\lambda$; and
  \item[$(ii)$]  each vertex $v \in V(C_n)$ appears in at most $A(v)$ subgraphs.
  \end{itemize}

\vspace{0.1cm}
\item[\textbf{\textsc{Objective:}}] Minimize
$\sum_{v \in V(C_n)}
A(v)$, and the optimum is denoted $A(n, C, \Delta)$.\\
  \end{itemize}

When the request graph is restricted to belong to a subclass of graphs
$\mathcal{C}$ of the class of graphs with maximum degree at most $\Delta$,
then the optimum is denoted $A(n, C, \Delta, \mathcal{C})$. Obviously, for
any subclass of graph $\mathcal{C}$, $A(n, C, \Delta, \mathcal{C})\leq
A(n, C, \Delta)$.

%Note that any solution for regular degree $\Delta$ is also a solution for
%maximum degree $\Delta$. The idea lies in thinking about the \emph{worst
%case} set of requests, because intuitively we have to place the ADMs in
%the vertices \emph{before} knowing which is the set of requests.

In this article we solve the cases corresponding to $\Delta=2$ and
$\Delta=3$ (giving a conjecture for the case $C=4$), and give lower bounds
for the general case. The remainder of the article is structured as
follows: in Section \ref{sec:behavior} we give some properties of the
function $A(n,C,\Delta)$, to be used in the following sections. In Section
\ref{sec:D2} we focus on the case $\Delta=2$, giving a closed formula for
all values of $C$. In Section \ref{sec:D3} we study the case $\Delta = 3$,
solving all cases except the case $C=4$, for which we conjecture the
solution. Finally, Section \ref{sec:conc} is devoted to conclusions and
open problems.

%\section{Main results}

%Note by $ADM(\Delta=\delta=k, N, C)$ the optimal solution in a
%unidirectional ring of $N$ nodes with grooming factor $C$ when the
%set of requests is given by a $k$-regular graph.

\section{Behavior of $A(n,C,\Delta)$}
\label{sec:behavior} \vspace{-0.2cm}In this section we describe some
properties of the function $A(n,C,\Delta)$.
%The first lemma states simple but useful
%properties of $A(n,C,\Delta)$.
\vspace{-0.2cm}
\begin{lemma}
\label{lem:easy} The following statements hold:
\begin{itemize}
\item[(i)] $A(n,C,1) = n$.
\item[(ii)] $A(n,1,\Delta) = \Delta n$.
\item[(iii)] If $C' \geq C$, then $A(n,C',\Delta) \leq A(n,C,\Delta)$.
\item[(iv)] If $\Delta'\geq \Delta$, then $A(n,C,\Delta') \geq A(n,C,\Delta)$.
\item[(v)] $A(n,C,\Delta) \geq n$ for all $\Delta \geq 1$.
\item[(vi)] If $C \geq \frac{n \Delta}{2}$, $A(n,C,\Delta)=n$.
\end{itemize}
\end{lemma}
\begin{proof}
\begin{itemize}
\item[\emph{(i)}] The request graph can consist in a perfect matching, so any solution uses 1 ADM
per node.
\item[\emph{(ii)}] A $\Delta$-regular graph can be partitioned into $\frac{n \Delta}{2}$
disjoint edges, and we cannot do better.
\item[\emph{(iii)}] Any solution for $C$ is also a solution for $C'$.
\item[\emph{(iv)}] If $\Delta'\geq \Delta$, the subgraphs with maximum degree at most $\Delta$ are a
subclass of the class of graphs with maximum degree at most $\Delta'$.
\item[\emph{(v)}] Combine \emph{(i)} and \emph{(iv)}.
\item[\emph{(vi)}] In this case all the edges of the request graph fit into one subgraph.
\end{itemize}
\vspace{-.5cm}
\end{proof}

Since we are interested in the number of ADMs required at each node, let
us consider the following definition:

\begin{definition}
Let $M(C,\Delta)$
%and $\alpha(C,\Delta)$
be the least positive number $M$ such that, for any $n\geq 1$, the
inequality $A(n,C,\Delta)\leq M n$ %-\alpha $
holds.
\end{definition}

\begin{lemma}
\label{lem:integer} $M(C,\Delta)$
%and $\alpha(C,\Delta)$ are
is a natural number.
\end{lemma}
\begin{proof}
First of all, we know by Lemma \ref{lem:easy} that, for any $C\geq 1$,
$A(n,C,\Delta) \leq A(n,1,\Delta) = \Delta n$. Thus $A(n,C,\Delta)$ is
upper-bounded by $\Delta n$. On the other hand, since any vertex may
appear in the request graph, $A(n,C,\Delta)$ is lower-bounded by $n$.

Suppose now that $M$ is not a natural number. That is, suppose that $r < M
< r+1$ for some positive natural number $r$. This means that, for each
$n$, there exists at least a fraction $\frac{r}{M}$ of the vertices with
at most $r$ ADMs. For each $n$, let $V_{n,r}$ be the subset of vertices of
the request graph with at most $r$ ADMs. Then, since $\frac{r}{M}>0$, we
have that $\lim_{n \rightarrow \infty}|V_{n,r}|=\infty$. In other words,
there is an arbitrarily big subset of vertices with at most $r$ ADMs per
vertex. But we can consider a request graph with maximum degree at most
$\Delta$ on the set of vertices $V_{n,r}$, and this means that with $r$
ADMs per node is enough, a contradiction with the optimality of $M$.
\end{proof}
%Finally, $\alpha$ is an integer because of the integrality of $M$, $n$,
%and $A(n,C,\Delta)$, and it is positive because if $\alpha$ was negative,
%then we could take an arbitrarily big subset of vertices using exactly $M$
%ADMs per node and arrive again at contradiction.
%\end{proof}\vspace{-0.5cm}\begin{definition}
%Let $M(C,\Delta)$ and $\alpha(C,\Delta)$ be the positive integers given by
%Lemma \ref{lem:integer}.
%\end{definition}
%\vspace{-0.05cm} Intuitively, this means that each vertex must appear in
%at most $M(C,\Delta)$ subgraphs. Henceforth, we will focus on finding the
%values of $M(C, \Delta)$ and $\alpha (M,\Delta)$, mainly the former
%because is the term that will give the dominant cost. Again,
If the
request graph is restricted to belong to a subclass of graphs
$\mathcal{C}$ of the class of graphs with maximum degree at most $\Delta$,
then %$A(n, C, \Delta, \mathcal{C})$ also satisfies Lemma \ref{lem:integer} and
the corresponding positive integer is denoted $M(C, \Delta,
\mathcal{C})$.
 Again, for any subclass $\mathcal{C}$, $M(C, \Delta,
\mathcal{C})\leq M( C, \Delta)$.
%The same remarks apply to
%$\alpha(C,\Delta, \mathcal{C})$.

Combining Lemmas \ref{lem:easy} and \ref{lem:integer}, we know that
$M(C,\Delta)$ decreases by integer hops when $C$ increases.
%, until meeting
%the lower bound of 1.
One would like to have a better knowledge of those
hops. The following lemma gives a sufficient condition to assure than
$M(C,\Delta)$ decreases by at most 1 when $C$ increases by 1.
\begin{lemma}
\label{lem:one_by_one} If $C > \Delta $, then $M(C+1,\Delta) \geq
M(C,\Delta)-1$.
\end{lemma}
\begin{proof}
Suppose that $M(C+1,\Delta)\leq M(C,\Delta)-2$, and let us arrive at
contradiction. Beginning with a solution for $C+1$, we will see that
adding $n$ ADMs (i.e. increasing $M$ by 1) we obtain a solution for $C$, a
contradiction with the assumption $M(C,\Delta)\geq M(C+1,\Delta)+2$.

The request graph has at most $\frac{\Delta n}{2}$ edges, and then in a
solution for $C+1$ the number of subgraphs with exactly $C+1$ edges is at
most $\frac{\Delta n}{2(C+1)}$. All the subgraphs with $C$ or less edges
can also be used in a solution for $C$. We remove an edge from each one of
the subgraphs with $C+1$ edges, obtaining at most $\frac{\Delta
n}{2(C+1)}$ edges, or equivalently at most $\frac{\Delta n}{C+1}$
additional ADMs. We want this number to be at most $n$, i.e.
$$\frac{\Delta n}{C+1} \leq n\ ,$$
which is equivalent to $\Delta \leq C+1$, and this is true by hypothesis.
\end{proof}
%\subsection{General Lower Bound}

%Si mirem be la fita que vem trobar a partir del grau mig: $M \geq
%\frac{C+1}{2C}\Delta$ es immediat que $M \geq \frac{\Delta}{2}$ per
%qualsevol valor de $C$. Per tant a la taula aquella $M$ mai no baixarà de
%$\frac{\Delta}{2}$

%Si tenim en compte el nombre de vertexs del graf de requests, es facil
%veure que $M \leq \frac{N}{2C}\Delta$

\noindent We provide now a lower bound on $M(C,\Delta)$.

%\iffalse
%\paragraph{Taking into account the number of edges}

%The main fact that obstructs the number of ADMs to fall down is that we
%cannot assure the existence of cycles of length smaller or equal than $C$.
%In this case, the best graphs we can use are trees with $C$ edges.
%\begin{proposition}
%\label{prop: bound2} If the set of requests is given by a $\Delta$-regular
%graph of girth greater than $C$, then
%$$
%A(n, C, \Delta) \geq \PartIntSup{\frac{C+1}{C} \frac{\Delta}{2} n}
%$$
%\end{proposition}
%\begin{proof}
%The proof consists just in counting how many ADMs would be needed in a
%decomposition of the request graph into trees with $C$ edges, which is the
%best possible. Since the number of edges is $\frac{n \Delta}{2}$, the
%number of trees is $\frac{n \Delta}{2C}$, and this yields the result
%taking into account that each tree has $C+1$ vertices.
%\end{proof}

%
%\paragraph{Taking into account the average degree}
%\fi

\begin{proposition}[General Lower Bound]
\label{prop: bound1} $ M(C,\Delta)\geq \PartIntSup{\frac{C+1}{C}
\frac{\Delta}{2}}. $
\end{proposition}
\begin{proof}
%By Lemma \ref{lem:integer}, $A(n,C,\Delta)$ is of the form
%$A(n,C,\Delta)=M n-\alpha$, where $M$ and $\alpha$ are natural numbers
% depending only on $C$ and $\Delta$. This means that each vertex must
%appear in at most $M(C,\Delta)$ subgraphs.
Since we have to consider all the graphs with maximum degree at most
$\Delta$, we can restrict ourselves to $\Delta$-regular graphs with girth
greater than $C$. Then, the best one could do is to partition the edges of
the request graph into trees with $C$ edges. In this case, the sum of the
degrees of all the vertices in each subgraph is $2C$. Thus, the average
degree of the vertices in all the subgraphs is at most $\frac{2C}{C+1}$,
hence it exists at least one vertex $v$ with average degree not greater
than $\frac{2C}{C+1}$. Therefore, $v$ must appear in at least $M_v$
subgraphs, with $\frac{2C}{C+1} M_v \geq \Delta$. Thus, $M(C,\Delta) \geq
\PartIntSup{\frac{C+1}{C} \frac{\Delta}{2}}$.
\end{proof}
%\vspace{-0.5cm} In particular, Proposition \ref{prop: bound1} implies that
%$M(C, \Delta) \geq \PartIntSup{\frac{\Delta}{2}}$, which can be
%arbitrarily big. Observe that this bound is not in contradiction with
%\emph{(vi)} of Lemma \ref{lem:easy}, because $\alpha(C,\Delta)$ may
%increase when $C$ increases, keeping $A(n,C,\Delta)$ equal to $n$ when $C
%\geq \frac{n \Delta}{2}$.

\begin{corollary}\label{corodelta}
If the set of requests is given by a $\Delta$-regular graph of girth
greater than $C$, then
$$
M(C, \Delta) \geq \PartIntSup{\frac{\Delta}{2}}
$$
\end{corollary}
\begin{proof}
Trivial from Proposition \ref{prop: bound1}.
\end{proof}

%\subsection{Combined lower bound}

%\begin{corollary}
%$$
%A(n, C, \Delta) \geq \max \left\{    \PartIntSup{\frac{2 \Delta}{C}}n,\
%\PartIntSup{\frac{C+1}{C} \frac{\Delta}{2}n} \right\}
%$$
%\end{corollary}
%\begin{proof}
%Just combine Proposition \ref{prop: bound1} and Proposition \ref{prop:
%bound2}.
%\end{proof}

%\iffalse \noindent It is easy to check that the lower bound of Proposition
%\ref{prop: bound1} is tighter than the lower bound of Proposition
%\ref{prop: bound2}.

%\paragraph{Open questions...} A natural question is: when can we assure the
%existence of cycles
%of length smaller or equal than $C$? (i.e. the \textbf{girth} of the request graph)\\

%Let $R$ be the request graph, as usual. If $n \geq |E(R)|$ then we know
%that there are cycles (because $R$ is not a tree in this case), but these
%cycles can be very large, until length $n$. Thus, a sufficient condition
%is that $C\geq n \geq |E(R)|$.

%\fi

%\subsection{General Upper Bound}

If the value of $C$ is large in comparison to $n$ the number of ADMs required per node may be
less than $M(C,\Delta)$ as stated in the following lemma:

%%%%%COMPTE AMB AIXOOOOOOOOOO
\begin{lemma}
\label{lem:upper} $ A(n,C,\Delta) \leq \PartIntSup{\frac{n \Delta}{2C}}n.$
\end{lemma}
\begin{proof}
%By Lemma \ref{lem:integer}, $A(n,C,\Delta)$ is of the form
%$A(n,C,\Delta)=M n-\alpha$, where $M$ and $\alpha$ are natural numbers
% depending only on $C$ and $\Delta$.
The number of edges of a request graph with degree $\Delta$ is at most $\frac{n \Delta}{2}$. We can
partition this edges greedily into subsets of at most $C$ edges, obtaining
at most $\PartIntSup{ \frac{n \Delta}{2C}}$ subgraphs. Thus, in this
partition each vertex appears in at most $\PartIntSup{ \frac{n
\Delta}{2C}}$ subgraphs, as we wanted to prove.
\end{proof}

Notice that this is not in contradiction with Corollary \ref{corodelta},
since the inequality of the definition of $M(C,\Delta)$ must hold for
\emph{all} values of $n$.

%\vspace{-0.45cm} Note that we claimed that $M(C,\Delta)$ does not depend on $n$,
%whereas the bound of Lemma \ref{lem:upper} seems to state that $M(C,\Delta) \leq \PartIntSup{\frac{n \Delta}{2C}}$.
%This is not contradictory, since, al

%since since although the exact

%value of $M(C,\Delta)$ is independent of $n$, it is possible to bound its
%value in terms of $n$. This bound is useful when $n$ is small related to
%$C$. \vspace{-0.1cm}
\section{Case $\Delta=2$}
\label{sec:D2}
%Placing ourselves in the worst case, consider degree $2$ in all the
%vertices (2-regular graphs of requests). In this case the set of requests
%is made of disjoint cycles. As we have explained in the previous section,
%the idea is to consider all the possible request graphs. Then, besides the
%small values of $n$, we have that

\begin{proposition}
\label{prop:degree2}
%If the set of requests is given by a $2$-regular
%graph of girth greater than $C$, then
$A(n, C, 2)=2n-(C-1).$
\end{proposition}
\begin{proof}
Consider the case when the request graph is $2$-regular and has girth
greater than $C$. Then, a feasible solution is obtained by placing $2$
ADMs at each vertex. What we do is to count in how many ADMs we can assure
that we can place only one
ADM.\\
Let us see first that we cannot use $1$ ADM in more than $C-1$ vertices.
Suppose this, i.e. that we have $1$ ADM in $C$ vertices and $2$ in all the
others. Then, consider a set of requests given by a cycle $H$ of length
$C+1$ containing all the $C$ vertices with $1$ ADM inside it, and other
cycles containing the remaining vertices. In this situation, we are forced
to use $2$ subgraphs for the vertices of $H$, and at least $2$ vertices of
$H$ must appear in both subgraphs. Hence we will need more than $1$ ADM in
some vertex that had initially only $1$ ADM.

Now, let us see that we can always save $C-1$ ADMs. Let
$\{a_0,a_1,\ldots,a_{C-2}\}$ be the set of vertices with only $1$ ADM,
that we can choose arbitrarily. We will see that we can decompose the set
of requests in such a way that the vertices $a_i$ always lie in the middle
of a path or a cycle, covering in this way both requests of each vertex
with only $1$ ADM. Indeed, suppose first that two of these vertices
(namely, $a_i$ and $a_j$) do not appear consecutively in one of the
disjoint cycles of the set of requests. Let $b_i$ be the nearest vertex to
$a_i$ in the cycle in the direction of $a_j$, and conversely for $b_j$
($b_i$ may be equal to $b_j$ if $a_i$ and $a_j$ differ only on one
vertex). Then, consider two paths (eventually, cycles) of the form
$\{b_i,a_i,\ldots\}$ and $\{b_j,a_j,\ldots\}$, to assure that both $a_i$
and $a_j$ lie in the middle of the subgraph. We do the same construction
for each pair of non-consecutive vertices.

Now, consider all the vertices $\{a_0,\ldots,a_i,\ldots,a_{t-1}\}$ which
are adjacent in the same cycle of the request graph, with $t\leq C-1$.
%Otherwise, suppose
%that all $C-1$ vertices lie consecutive in a cycle of the request graph.
Let $b_0$ be the nearest vertex to $a_0$ different from $a_1$, and let
$b_{t-1}$ be the nearest vertex to $a_{t-1}$ different from $a_{t-2}$.
Then, consider a subgraph with the path (or cycle, if $b_0=b_{t-1}$)
$\{b_0a_0a_1\ldots a_{t-1}b_{t-1}\}$.
\end{proof}

\vspace{-0.8cm}
\section{Case $\Delta=3$}
\label{sec:D3} We study the cases $C=3$ and $C\geq 5$ in Sections
\ref{sec:C3} and \ref{sec:5,3}, respectively. We discuss the open case
$C=4$ in Section \ref{sec:conc}.
\subsection{Case $C=3$}
\label{sec:C3} We study first the case when the request graph is a
bridgeless cubic graph in Section \ref{sec:bridgeless}, and then the case
of a general request graph in Section \ref{sec:general}.
\subsubsection{Bridgeless Cubic Request Graph}
\label{sec:bridgeless}
 We will need some preliminary graph theoretical concepts. Let
$G=(V,E)$ be a graph. For $A,B\subseteq V$, an \emph{A-B path} in $G$ is a
path from $x$ to $y$, with $x \in A$ and $y \in B$.
%\begin{definition}[Separate]
If $A,B\subseteq V$ and $X\subseteq V\cup E$ are such that every
\emph{A-B} path in $G$ contains a vertex or an edge from $X$, we
say that $X$ \emph{separates} the sets $A$ and $B$ in $G$. More
generally we say that $X$ \emph{separates} $G$ if $G-X$ is
disconnected, that is, if $X$ separates in $G$ some two vertices
that are not in $X$. A separating set of vertices is a
\emph{separator}.
%\end{definition}
%\begin{definition}[Cut-vertex; Bridge]
A vertex which separates two other vertices of the same component is a
\emph{cut-vertex}, and an edge separating its ends is a \emph{bridge}.
Thus, the bridges in a graph are precisely those edges that do not lie on
any cycle.
%\end{definition}
%A graph with a bridge is shown in \textsc{Fig.} \ref{fig:bridge}. The
%end-vertices of the bridge are the cut-vertices.
%\begin{figure}[h!tb]
%\begin{center}
%\includegraphics[width=4.5cm]{Fig/bridge}
%\caption{Example of a graph with a bridge} \label{fig:bridge}
%\end{center}
%\end{figure}
%\begin{definition}[Matching]
A set $M$ of independent edges in a graph $G=(V,E)$ is called a
\emph{matching}.
%$M$ is a matching of $U\subseteq V$ if every
%vertex in $U$ is incident with an edge in $M$. The vertices in $U$
%are then called \emph{matched} (by $M$); vertices not incident
%with any edge of $M$ are \emph{unmatched}.
%\end{definition}
%\begin{definition}[$k$-factor]
A $k$-regular spanning subgraph is called a $k$-factor. Thus, a subgraph
$H\subseteq G$ is a 1-factor of $G$ if and only if $E(H)$ is a matching of
$V$.
%\end{definition}
We recall a well known result from matching theory:
\begin{theorem}[Petersen, 1981]
Every bridgeless cubic graph has a 1-factor.
\end{theorem}
Then, if we remove a 1-factor from a cubic graph, what it remains
is a disjoint set of cycles.
\begin{corollary}
\label{prop:bridgeless} Every bridgeless cubic graph has a
decomposition into a 1-factor and disjoint cycles.
\end{corollary}
An example of a decomposition of a bridgeless cubic graph into disjoints
cycles and a 1-factor is depicted in \textsc{Fig.} \ref{fig:1factor}a.

\begin{figure}[h!tb]
\begin{center}
\vspace{-0.5cm}
\includegraphics[width=14.5cm]{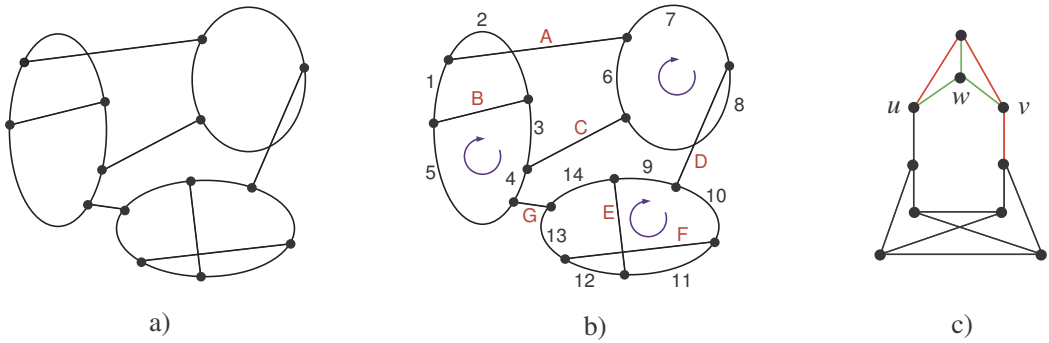}
\caption{\textbf{a)} Decomposition of a bridgeless cubic graph into
disjoints cycles and a 1-factor. \textbf{b)} Decomposition of a bridgeless
cubic graph into paths of length $3$. \textbf{c)} Cubic bridgeless graph
used in the proof of Proposition \ref{prop:degree3}} \label{fig:1factor}
\vspace{-0.5cm}
\end{center}
\end{figure}

%We will use Proposition \ref{prop:bridgeless} in the following
%result.
\begin{proposition}
\label{prop:degree3} Let $\mathcal{C}$ be the class of bridgeless cubic
graphs. Then,
$$
M(3,3, \mathcal{C})=2.
$$
\end{proposition}
\begin{proof}
Let us proof that we can always partition the request graph into paths
with 3 edges in such a way that each vertex appears in 2 paths. To do so,
we take the decomposition given by Proposition \ref{prop:bridgeless},
together with a clockwise orientation of the edges of each cycle. With
this orientation, each edge of the 1-factor has two \emph{incoming} and
two \emph{outgoing} edges of the cycles. For each edge of the 1-factor we
take its two incoming edges, and form in this way a path of length $3$. It
is easy to verify that this is indeed a decomposition into paths of length
three. For instance, if we do this construction in the graph of
\textsc{Fig.} \ref{fig:1factor}a, and we label the edges of the 1-factor
as \{A,B,\ldots,G\} and the ones of the cycles as \{1,2,\ldots,14\} (see
\textsc{Fig.} \ref{fig:1factor}b), we obtain the following decomposition:
$$
\{1,A,6\},\{5,B,2\},\{3,C,8\},\{7,D,9\},\{14,E,11\},\{10,F,12\},\{4,G,13\}
$$
%\begin{figure}[h!tb]
%\begin{center}
%\includegraphics[width=6cm]{bridge_dec}
%\caption{Decomposition of a bridgeless cubic graph into paths of length
%$3$} \label{fig:brige_dec}
%\end{center}
%\end{figure}
Now let us see that we cannot do better, i.e. with $2n-1$ ADMs. If such a
solution exists, there would be at least one vertex with only $1$ ADM, and
the average of the number of ADMs of all the other vertices must not
exceed $2$. In order to see that this is not always possible, consider the
cubic bridgeless graph on $10$ vertices of \textsc{Fig.}
\ref{fig:1factor}c. Let $w$ be the vertex with only $1$ ADM. This graph
 has no triangles except those containing $w$.
Since we can use only $1$ ADM in $w$, we must take all its requests in one
subgraph. It is not possible to cover the $4$ remaining requests of the
nodes $u$ and $v$ in one subgraph,
%(the best we can do is to take the red path)
 and thus without loss of generality we will need $3$ ADMs in $u$. With these constraints,
one can check that the best solution uses $20$ ADMs, that is $2n>2n-1$.
%\begin{figure}[h!tb]
%\begin{center}
%\vspace{-0.5cm}
%\includegraphics[width=2.7cm]{badgraph2}
%\caption{Cubic bridgeless graph used in the proof of Proposition
%\ref{prop:degree3}} \label{fig:badgraph} \vspace{-0.5cm}
%\end{center}
%\end{figure}
%\vspace{-0.5cm}
\end{proof}
\vspace{-0.5cm} Taking a look at the proof we see that the only property
that we need from the bridgeless cubic graph is that we can partition it
into a 1-factor and disjoint cycles. Hence, we can relax the hypothesis of
Proposition \ref{prop:degree3} to obtain the following corollary:
\begin{corollary}
\label{prop:degree3bis} Let $\mathcal{C}$ be the class of graphs of
maximum degree at most 3 that can be partitioned into disjoints cycles and
a 1-factor. Then
\begin{itemize}
\item[(i)] $M(3, 3, \mathcal{C})=2$; and
\item[(ii)] $M(C, 3,\mathcal{C})\leq 2$ for any $C\geq4$.
\end{itemize}
\end{corollary}
%\begin{proof}
%We can do the same construction of Proposition \ref{prop:degree3} to
%obtain $2n$ ADMs.
%\end{proof}

%The main result of this section is Theorem \ref{teo:paths3}, that states
%that a graph with maximum degree at most three can be partitioned into
%trees with at most three edges.

%\newpage

\subsubsection{General Request Graph}
\label{sec:general} \noindent It turns out that when the request graph is
not restricted to be bridgeless we have that $M(3,3)=3$.

\label{sec:3,3}
\begin{proposition}
\label{prop:counter}
 $M(3,3)=3$.
\end{proposition}
\begin{proof}
By \emph{(ii)} and \emph{(iii)} of Lemma \ref{lem:easy} we know that
$M(3,3)\leq 3$. We shall exhibit a counterexample showing that $M(3,3)>
2$, proving the result. Consider the cubic graph $G$ depicted in
\textsc{Fig. }\ref{fig:counter3}a. We will prove that it is not possible
to partition the edges of $G$ into subgraphs with at most 3 edges in such
a way that each vertex appears in at most 2 subgraphs.

\begin{figure}[h!tb]
\begin{center}
\vspace{-0.5cm}
\includegraphics[width=16cm]{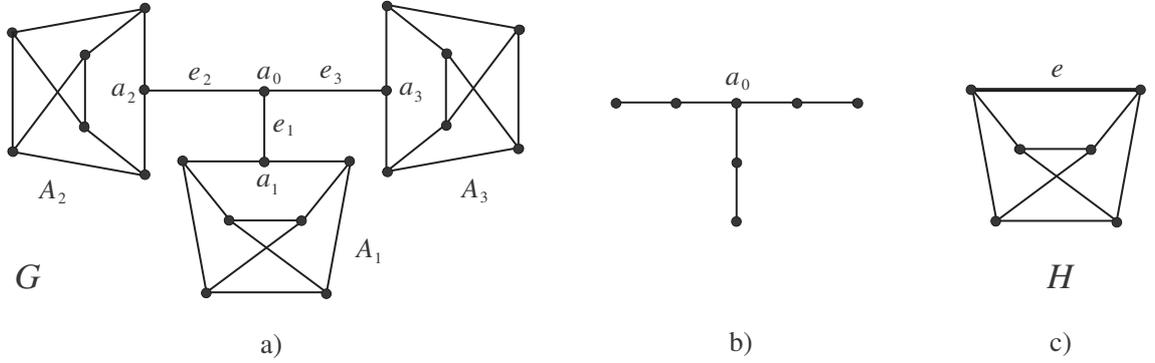}
\caption{\textbf{a)} Cubic graph $G$ that can not be edge-partitioned into
subgraphs with at most 3 edges in such a way that each vertex appears in
at most 2 subgraphs. \textbf{b)} Graph that cannot be partitioned into 2
connected subgraphs with at most 3 edges. \textbf{c)} Counterexample of
Proposition \ref{prop:counter} showing that $M(3,3)=3$}
\label{fig:counter3} \vspace{-0.5cm}
\end{center}
\end{figure}

Indeed, suppose the opposite, i.e. that we can partition the edges of $G$
into (connected) subgraphs $B_1,\ldots,B_k$ with $|E(B_i)| \leq 3$ in such
a way that each vertex appears in at most 2 subgraphs, and let us arrive
at a contradiction.

Following the notation illustrated in \textsc{Fig. }\ref{fig:counter3}a,
let $A_1,A_2,A_3$ be the connected components of $G \setminus
\{e_1,e_2,e_3\}$. Let also, with abuse of notation, $a_i=A_i \cap e_i$,
$i=1,2,3$, and $a_0 = e_1 \cap e_2 \cap e_3$.
\begin{claim}
\label{claim:1} There exist an index $i^* \in \{1,2,3\}$ and a subgraph
$B_{k^*}$ containing $a_0$,
 such that $B_{k^*} \cap A_{i^*} = \{a_{i*}\}$.
\end{claim}
\vspace{-.5cm}
\begin{proof}
Among all the subgraphs $B_1,\ldots,B_k$ involved in the decomposition of
$G$, consider the $\ell$ subgraphs $B_{j_1},\ldots,B_{j_\ell}$ covering
the edges $\{e_1,e_2,e_3\}$. If $\ell=1$, then the subgraph $B_{j_1}$ is a
star with three edges and center $a_0$, and then $B_{j_1} \cap A_i =
\{a_i\}$ for each $i=1,2,3$. If $\ell \geq 3$, then the vertex $a_0$
appears in 3 subgraphs, a contradiction. Hence it remains to handle the
case $\ell=2$. If the claim was not true, it would imply that for each
$i=1,2,3$ it would exist $j_{f(i)}\in \{j_1,j_2\}$ such that $B_{f(i)}
\cap A_{i}$ contains at least one edge. In particular, this would imply
that the graph depicted in \textsc{Fig. }\ref{fig:counter3}b could be
partitioned into 2 connected subgraphs with at most 3 edges, which is
clearly not possible.
\end{proof}
\vspace{-.6cm} Suppose without loss of generality that the index $i^*$
given by Claim \ref{claim:1} is equal to 1. Thus, $a_{1}$ appears in a
subgraph $B_{k^*}$ that does not contain any edge of $A_{1}$. Therefore,
the edges of $A_{1}$ must be partitioned into connected subgraphs with at
most 3 edges, in such a way that $a_{1}$ appears in only 1 subgraph, and
all its other vertices in at most 2 subgraphs. Let us now see that this is
not possible, obtaining the contradiction we are looking for.

Indeed, since $a_{1}$ has degree 2 in $A_{1}$ and it can appear in only
one subgraph, it must have degree two in the subgraph in which it appears,
i.e. in the middle of a $P_3$ or a $P_4$, because $A_{1}$ is
triangle-free. It is easy to see that this is equivalent to partitioning
the edges of the graph $H$ depicted in \textsc{Fig. }\ref{fig:counter3}c
into connected subgraphs with at most 3 edges, in such a way that the
thick edge $e$ appears in a subgraph with at most 2 edges, and each vertex
appears in at most 2 subgraphs. Observe that $H$ is cubic and
triangle-free. Let $n_1$ be the total number of vertices of degree 1 in
all the subgraphs of the decomposition of $H$. Since each vertex of $H$
can appear in at most 2 subgraphs and $H$ is cubic, each vertex can appear
with degree 1 in at most 1 subgraph. Thus, $n_1 \leq |V(H)|= 6$.

Since we have to use at least 1 subgraph with at most 2 edges and
$|E(H)|=9$, there are at least $1+\PartIntSup{\frac{9-2}{3}}=4$ subgraphs
in the decomposition of $H$. But each subgraph involved in the
decomposition of $H$ has at least 2 vertices of degree 1, because $H$ is
triangle-free. Therefore, $n_1 \geq 8$, a contradiction.
%Summing up, $8\leq n_1 \leq 6$, a contradiction.
\end{proof}

\vspace{-0.5cm}

\subsection{Case $C\geq 5$}
\label{sec:5,3}

For $C\geq 5$ we can easily prove that $M(C,3)=2$, making use of a
conjecture made by Bermond \emph{et al.} in 1984 \cite{BFHP84} and proved
by Thomassen in 1999 \cite{Th99}:
\begin{theorem}[\cite{Th99}]
\label{theo:tho} The edges of a cubic graph can be 2-colored such that
each monochromatic component is a path of length at most 5.
\end{theorem}

A \emph{linear $k$-forest} is a forest consisting of paths of length at
most $k$. The \emph{linear $k$-arboricity} of a graph $G$ is the minimum
number of linear $k$-forests required to partition $E(G)$, and is denoted
by $la_k(G)$ \cite{BFHP84}. Theorem \ref{theo:tho} is equivalent to saying
that, if $G$ is cubic, then $la_2(G)=2$.

Let us now see that Theorem \ref{theo:tho} implies that $M(C,3)=2$ for all
. Indeed, all the paths of the linear forests have at most $5$ edges, and
each vertex will appear in exactly $1$ of the $2$ linear forests, so the
decomposition given by Theorem \cite{Th99} is a partition of the edges of
a cubic graph into subgraphs with at most $5$ edges, in such a way that
each vertex appears in at most $2$ subgraphs. In fact the result of
\cite{Th99} is stronger, in the sense that $G$ can be any graph of maximum
degree at most 3. Thus, we deduce that \vspace{-0.05cm}
\begin{corollary}
For any $C\geq 5$, $M(C,3)=2$.
\end{corollary}
\vspace{-0.05cm} Thomassen also proved \cite{Th99} that 5 cannot be
replaced by 4 in Theorem \ref{theo:tho}. This fact do not imply that
$M(4,3)=3$, because of the following reasons: (i) the subgraphs of the
decomposition of the request graph are not restricted to be paths, and
(ii) it is not necessary to be able to find a 2-coloring of the subgraphs
of the decomposition (a \emph{coloring} in this context means that each
subgraph receives a color, and 2 subgraphs with the same color must have
empty intersection).

%This result is stronger that the fact , as it is shown by the following
%example: (cubic graph on 6 vertices...)

\section{Conclusions}
\label{sec:conc} We have considered the traffic grooming problem in
unidirectional WDM rings when the request graph belongs to the class of
graph with maximum degree $\Delta$. This formulation allows the network to
support dynamic traffic without reconfiguring the electronic equipment at
the nodes. We have formally defined the problem, and we have focused
mainly on the cases $\Delta =2$ and $\Delta = 3$, solving completely the
former and solving all the cases of the latter, except the case when the
grooming value $C$ equals $4$. We have proved in Section \ref{sec:3,3}
that $M(3,3)=3$, and in Section \ref{sec:5,3} that $M(C,3)=2$ for all
$C\geq 5$. Because of the integrality of $M(C,\Delta)$ and Lemma
\ref{lem:easy}, $M(4,3)$ equals either 2 or 3. We conjecture that

\vspace{-0.15cm}
\begin{conjecture}
\label{conj:4,3} The edges of a graph with maximum degree at most 3 can be
partitioned into subgraphs with at most 4 edges, in such a way that each
vertex appears in at most 2 subgraphs.
\end{conjecture}
\vspace{-0.15cm}

If Conjecture \ref{conj:4,3} is true, it clearly implies that $M(4,3)=2$.
Corollary \ref{prop:degree3bis} states that $M(4,3,\mathcal{C})=2$,
$\mathcal{C}$ being the class of bridgeless graphs of maximum degree at
most 3. Nevertheless, finding the value of $M(4,3)$ remains open.
%For this case we have conjectured the solution in Conjecture
%\ref{conj:4,3}.
We have also deduced lower and upper bounds in the general case (any value
of $C$ and $\Delta$). \textsc{Tab.} \ref{tab:results} summarizes the
values of $M(C,\Delta)$ that we have obtained.

\begin{table}[h!tbp]
\begin{center}
$$
\begin{array}{|c||c|c|c|c|c|c|c|c|}

  \hline
  % after \\: \hline or \cline{col1-col2} \cline{col3-col4} ...
  C\ \setminus \ \Delta & 1 & 2  & 3 & 4 & 5 & 6 & \ldots & \Delta\\
  \hline
  \hline
  1 & 1 & 2  & 3 & 4 & 5 & 6 & \ldots & \Delta \\
  \hline
  2 & 1 &2  & 3 & 4 & 5 & 6 & \ldots & \Delta  \\
  \hline
  3 & 1 &2  & 3 & \geq 3 & \geq 4 & \geq 4 & \ldots & \geq \PartIntSup{\frac{2 \Delta}{3}} \\
  \hline
  4 & 1 &2  & 2 ?? & \geq 3 & \geq 4 & \geq 4 & \ldots & \geq \PartIntSup{\frac{5 \Delta}{8}}\\
  \hline
   5 & 1 &2  & 2 & \geq 3 & \geq 3 & \geq 4 & \ldots & \geq \PartIntSup{\frac{3 \Delta}{5}}\\
  \hline
  \ldots  &1& \ldots  & \ldots & \ldots & \ldots & \ldots & \ldots & \ldots \\
  \hline
  C & 1 & 2 &   2 & \geq \PartIntSup{\frac{C+1}{C} 2} & \geq \PartIntSup{\frac{C+1}{C} \frac{5}{2}}  & \geq \PartIntSup{\frac{C+1}{C} 3} & \ldots & \geq \PartIntSup{\frac{C+1}{C} \frac{\Delta}{2}}\\
  \hline

\end{array}
$$
\caption{Values of $M(C,\Delta)$. The case $C=4,\Delta=3$ is a conjectured
value}\label{tab:results}
\end{center}
 \end{table}

This problem can find wide applications in the design of optical networks
using WDM technology. It would be interesting to continue the study for
larger values of $\Delta$, which will certainly rely on graph
decomposition results. Another generalization could be to restrict the
request graph to belong to other classes of graphs for which there exist
powerful decomposition tools, like graphs with bounded tree-width, or
families of graphs excluding a fixed graph as a minor.

\bibliographystyle{abbrv}
\addcontentsline{toc}{section}{References}
\bibliography{ADM}
%\include{bib}
\iffalse
\newpage
\begin{appendix}
\section{Proof of Lemma \ref{lem:one_by_one}}
\label{ap:1}

\section{Construction in the proof of Proposition \ref{prop:degree2}}
\label{ap:2}
\begin{figure}[h!tb]
\begin{center}
\includegraphics[width=5.5cm]{degree2}
\caption{Construction in the proof of Proposition \ref{prop:degree2}}
\label{fig:degree2}
\end{center}
\end{figure}
\section{Values of $M(C,\Delta)$}
\label{ap:3}
\end{appendix}
\fi
\end{document}